\documentclass[aihp]{imsart}

\RequirePackage{amsthm,amsmath,amsfonts,amssymb}
\RequirePackage[numbers]{natbib}

\usepackage{amscd}
\usepackage{amsthm,amsmath,amsfonts}
\usepackage[utf8]{inputenc}
\usepackage{amssymb}
\usepackage{stmaryrd}
\usepackage{wasysym}
\usepackage{mathrsfs}
\usepackage{hyperref} %uncomment this and comment below bit to remove the page numbers in references
\usepackage{cleveref}
\usepackage{enumitem}
\usepackage{relsize} %for making things bigger in math mode, see https://tex.stackexchange.com/questions/490066/bigger-equation-in-text-mode-math
\usepackage{dsfont}
\usepackage{soul}
\usepackage{filecontents}
\hypersetup{colorlinks=true,linkcolor=black,citecolor=black}
\usepackage{color}
\usepackage[all]{xy}
\usepackage{float}
\usepackage{bm}
\usepackage{mathtools}
\usepackage{thmtools}
\usepackage{comment}
\usepackage{tikz}
\usepackage{ytableau}
\usepackage{mathdots}

\startlocaldefs
\numberwithin{equation}{section}
\theoremstyle{plain}

\newtheorem{thm}{Theorem}[section]

\newtheorem{prop}[thm]{Proposition}

\newtheorem{lemma}[thm]{Lemma}

\theoremstyle{definition}
\newtheorem{defi}{Definition}

\newtheorem{rmk}{Remark}

\DeclareSymbolFont{bbold}{U}{bbold}{m}{n}
\DeclareSymbolFontAlphabet{\mathbbold}{bbold}

\newcommand{\R}{\mathbb{R}}
\newcommand{\Z}{\mathbb{Z}}
\newcommand{\Q}{\mathbb{Q}}
\newcommand{\N}{\mathbb{N}}
\newcommand{\C}{\mathbb{C}}
\newcommand{\F}{\mathbb{F}}

\newcommand{\E}{\mathbb{E}}

\newcommand{\W}{\mathbb{W}}

\newcommand{\gl}{\mathfrak{gl}}

\newcommand{\bbone}{\mathbbold{1}}
\newcommand{\la}{\lambda}

\newcommand{\eps}{\epsilon}

\newcommand{\tV}{\tilde{V}}

\newcommand{\lan}{\left\langle}
\newcommand{\ran}{\right\rangle}

\newcommand{\tth}{^{th}}

\newcommand{\tnu}{{\tilde{\nu}}}

\newcommand{\sY}{\mathsf{Y}}
\newcommand{\y}{\mathsf{y}}

\newcommand{\bi}{\mathbf{i}}

\renewcommand{\vec}[1]{\boldsymbol{#1}}

\DeclareMathOperator{\Span}{span}
\DeclareMathOperator{\Tr}{Tr}

\DeclareMathOperator{\Sig}{Sig}

\DeclareMathOperator{\Proj}{Proj}

\DeclareMathOperator{\SN}{SN}

\DeclareMathOperator{\diag}{diag}

\DeclareMathOperator{\Mat}{Mat}
\DeclareMathOperator{\val}{val}

\DeclareMathOperator{\Law}{Law}

\newcommand{\Pois}{\mathcal{S}}

\newcommand{\tell}{{\tilde{\ell}}}
\newcommand{\tU}{\tilde{U}}
\newcommand{\GL}{\mathrm{GL}}
\newcommand{\SL}{\mathrm{SL}}

\newcommand{\U}{\mathrm{U}}

\endlocaldefs

\begin{document}

\begin{frontmatter}

\title{What is a $p$-adic Dyson Brownian motion?}
%\title{A sample article title with some additional note\thanksref{T1}}
\runtitle{What is a $p$-adic Dyson Brownian motion?}
%\thankstext{T1}{A sample of additional note to the title.}

\begin{aug}
\author[A]{\inits{RVP}\fnms{Roger}~\snm{Van Peski}\ead[label=e1]{rvp@mit.edu}}
% \author[B]{\inits{???}\fnms{???}~\snm{???}\ead[label=e2]{???@???}}
% \and
% \author[B]{\inits{???}\fnms{???}~\snm{???}\ead[label=e3]{???@???}}
%%%%%%%%%%%%%%%%%%%%%%%%%%%%%%%%%%%%%%%%%%%%%%
%% Addresses                                %%
%%%%%%%%%%%%%%%%%%%%%%%%%%%%%%%%%%%%%%%%%%%%%%

 \address[A]{Department of Mathematics, KTH Royal Institute of Technology, Stockholm, Sweden\printead[presep={,\ }]{e1}}

% \address[B]{???\printead[presep={,\ }]{e2,e3}}
\end{aug}

\begin{abstract}
We consider the singular numbers of a certain explicit continuous-time Markov jump process on $\mathrm{GL}_N(\mathbb{Q}_p)$, which we argue gives the closest $p$-adic analogue of multiplicative Dyson Brownian motion. We do so by explicitly classifying the possible dynamics of singular numbers of processes on $\mathrm{GL}_N(\mathbb{Q}_p)$ satisfying natural properties possessed by Brownian motion on $\mathrm{GL}_N(\mathbb{C})$. Computing the evolution of singular numbers explicitly, we find that the $N$-tuple of singular numbers in decreasing order evolves as a Poisson jump process on $\Z^N$, with ordering enforced by reflection off the walls of the positive type $A$ Weyl chamber. This contrasts with---and provides a $p$-adic analogue to---the behavior of classical Dyson Brownian motion, where ordering is enforced by conditioning to avoid the Weyl chamber walls.
\end{abstract}

% \begin{abstract}[language=french]
% (to be added in by editors)
% \end{abstract}



\end{frontmatter}

\maketitle

\tableofcontents

\section{Introduction}

\subsection{Preface} Around the turn of the millennium, a series of works (detailed later in \Cref{rmk:reflected_rw}) found that the following three are the same: 
\begin{enumerate}[label=(\Roman*)]
\item Eigenvalues and singular values of certain Brownian motions on spaces of matrices,\label{item:mbm}
\item Brownian motions on the positive Weyl chamber in each Lie type, conditioned to avoid the walls for all time, and \label{item:weyl}
\item Path transformations of Brownian motions coming from the Robinson-Schensted-Knuth (RSK) correspondence. \label{item:rsk}
\end{enumerate}
Dyson \cite{dyson1962brownian} introduced the first version of \ref{item:mbm}. The corresponding stochastic evolution of the eigenvalues, which may be viewed as a stochastic process on the positive type $A$ Weyl chamber, is now called \emph{Dyson Brownian motion}.

It has become apparent that many results on real and complex random matrices have interesting and nontrivially different analogues for $p$-adic random matrices. This work finds a version of the connection \ref{item:mbm} $\leftrightarrow$ \ref{item:weyl} on the $p$-adic side. The main result, \Cref{thm:dbm_poisson_intro}, shows that the analogue of Dyson Brownian motion---more specifically, its multiplicative analogue defined in the next subsection---is a random walk on the integer points of the positive Weyl chamber which \emph{reflects} off the walls, rather than being conditioned to avoid them as with Dyson's process. It is quite interesting to speculate over analogues of \ref{item:rsk} for this process, and we do so later in \Cref{rmk:reflected_rw}, but so far we have little to say.

\subsection{Dyson Brownian motion} Dyson \cite{dyson1962brownian} showed the following elegant and useful fact: let $A(T), T \in \R_{\geq 0}$ be a stochastic process on $N \times N$ Hermitian matrices where each above-diagonal entry evolves according to an independent standard complex Brownian motion (and the below-diagonal entries are determined by Hermiticity), and each diagonal entry by $\sqrt{2}$ times a standard real Brownian motion. Then the eigenvalues evolve as $N$ independent Brownian motions conditioned not to intersect for all time, or equivalently $N$ independent Brownian motions with certain Coulomb repulsion interactions, a stochastic process known as \emph{Dyson Brownian motion}. Said another way, the vector $(\la_1(T),\la_2(T),\ldots,\la_N(T))$ of eigenvalues in decreasing order evolves as a Brownian motion in $\R^N$ conditioned to remain in the positive type $A$ Weyl chamber\footnote{Strictly speaking, it is $\W_N/\R \cdot (1,\ldots,1)$ which is the positive Weyl chamber of type $A_{N-1}$, but we follow the standard abuse of terminology and refer to $\W_N$ as a Weyl chamber.} $\W_N := \{(x_1,\ldots,x_N) \in \R^N: x_1 \geq \ldots \geq x_N\}$ for all time. More recently, this process became a key tool for proving universality results in random matrix theory via the so-called heat flow method in work of Bourgade, Erd\"os, Schlein, Yau, Yin and others \cite{bourgade2014edge,bourgade2016fixed,erdos2010universality,erdHos2011universality,erdHos2012universality,erdHos2012bulk,erdHos2015gap}. 

A similar construction exists for singular values of non-Hermitian matrices in $\GL_N(\C)$. The \emph{multiplicative Brownian motion} $\sY^{(N)}(T)$ (with some initial condition) is the stochastic process with infinitesimal generator given by $\frac{1}{2}\Delta$, where $\Delta$ is the Laplace-Beltrami operator on $\GL_N(\C)$ with respect to the metric given by the Hilbert-Schmidt inner product $\lan X, Y \ran := \frac{1}{N}\Tr (X Y^*)$ on $Lie(\GL_N(\C)) = \gl_N(\C) = \Mat_N(\C)$. We recommend Jones-O'Connell \cite{jones2006weyl} for an excellent exposition of multiplicative Brownian motions in the cases of common matrix groups and relations to Brownian motions on Weyl chambers in general Lie type.

The reason for the name `multiplicative' is that multiplicative Brownian motion may be viewed as a limit of discrete-time matrix product random walks on $\GL_N(\C)$ reminiscent of Donsker's theorem, while the usual (additive) Brownian motion on Hermitian matrices described above is a limit of discrete-time random walks given by sums of random matrices. In the multiplicative case, if $A_1^{(\ell)},A_2^{(\ell)},\ldots$ are iid elements of $\GL_N(\C)$ with distributions concentrating around the identity as $\ell \to \infty$, then the discrete random walk $Y^{(N,\ell)}(i) := A_i^{(\ell)} \cdots A_1^{(\ell)}$ converges to $\sY^{(N)}(T)$ as $\ell \to \infty$ with $i$ scaled appropriately. See Stroock-Varadhan \cite[Theorem 2.4]{stroock1973limit} for a precise statement which applies in much greater generality. 

The simple description of eigenvalues via Dyson Brownian motion has an analogue in this multiplicative context. Namely, if $\y_1(T) \geq \y_2(T) \geq \ldots \geq \y_N(T)$ are the singular values of $\sY^{(N)}(T)$, then the vector $(\log \y_1(T),\ldots,\log \y_N(T))$ evolves as a standard Brownian motion in $\R^N$ with drift $(N-1,N-3,\ldots,3-N,1-N)$, conditioned to remain in $\W_N$ for all time. We therefore call this process \emph{multiplicative Dyson Brownian motion}; in the literature it is also often referred to as the radial part of a Brownian motion on the homogeneous space $\GL_N(\C)/\U(N)$. The present work considers $p$-adic analogues of this multiplicative setting rather than the additive one, as the former turns out to be simpler, see \Cref{rmk:additive_case} for further discussion.

%. The state space of Hermitian matrices is just the tangent space of $\GL_N(\C)/\U(N)$ (after multiplying by $\bi$), and 

\subsection{$p$-adic random matrix theory} Many works have considered discrete random matrices over the integers $\Z$ or $p$-adic integers $\Z_p$, which yield distributions on abelian groups appearing in arithmetic statistics and the theory of random graphs: see Friedman-Washington \cite{friedman-washington}, Ellenberg-Venkatesh-Westerland \cite{ellenberg2011modeling,ellenberg2016homological}, Bhargava-Kane-Lenstra-Poonen-Rains \cite{bhargava2013modeling}, Clancy-Kaplan-Leake-Payne-Wood \cite{clancy2015cohen}, Wood \cite{wood2015random,wood2017distribution,wood2018cohen}, Kovaleva \cite{kovaleva2020note}, Lipnowski-Sawin-Tsimmerman \cite{lipnowski2020cohen}, M\'esz\'aros \cite{meszaros2020distribution}, Cheong-Huang \cite{cheong2021cohen} -Kaplan \cite{cheong2022generalizations} and -Yu \cite{cheong2023cokernel}, Nguyen-Wood \cite{nguyen2022random,nguyen2022local}, and Lee \cite{lee2022universality,lee2023joint}. The $p$-adic case is simpler than the integer case, while retaining most interesting features. For the basic definitions and properties of $p$-adic numbers we refer to the first two paragraphs of \Cref{sec:p-adic}.

Singular value (also known as Cartan) decomposition works in the same way for $p$-adic matrices, after replacing $\GL_N(\C)$ and its maximal compact subgroup $\U(N)$ with $\GL_N(\Q_p)$ and its maximal compact subgroup $\GL_N(\Z_p)$. Namely, just as singular value decomposition tells that the double cosets of $\U(N) \backslash \GL_N(\C) / \U(N)$ are parametrized by weakly decreasing tuples $(\la_1,\ldots,\la_N) \in \R_+^N$ of singular values, Smith normal form tells that the double cosets of $\GL_N(\Z_p) \backslash \GL_N(\Q_p) / \GL_N(\Z_p)$ are parametrized by the set of \emph{integer signatures of length $N$},
\begin{equation*}
\Sig_N := \{(\la_1,\ldots,\la_N) \in \Z^N: \la_1 \geq \ldots \geq \la_N\}.
\end{equation*}
Concretely, for any $A \in \GL_N(\Q_p)$, there exist $U,V \in \GL_N(\Z_p)$ and $\la = (\la_1,\ldots,\la_N) \in \Sig_N$ for which
\begin{equation}\label{eq:snf}
A = U \diag(p^{\la_1},\ldots,p^{\la_N})V,
\end{equation}
and any two such decompositions may feature different $U,V$ but will still have the same signature $\la$. The integers $\la_i$ are known as the \emph{singular numbers}\footnote{We note that previous works \cite{assiotis2022infinite,bufetov2017ergodic,neretin2013hua} use the opposite sign convention on singular numbers; we use the above one so that they are nonnegative when $A \in \Mat_N(\Z_p)$.} of $A$, and when $A$ has decomposition as in \eqref{eq:snf} we use notation
\begin{equation*}
    \SN(A) := \la
\end{equation*}
for the $N$-tuple of singular numbers in weakly decreasing order. The singular numbers are the analogues of the logarithms of singular values in the complex setting, though a key difference which will be relevant later is that $\Sig_N$ is a discrete set.

Some previous works in $p$-adic random matrix theory by Neretin \cite{neretin2013hua}, Bufetov-Qiu \cite{bufetov2017ergodic}, Assiotis \cite{assiotis2022infinite}, and the author \cite{van2020limits,vanpeski2021halllittlewood}, which come from a more Lie-theoretic standpoint than those mentioned above, have found close structural analogies between singular values of complex matrices and their analogues for $p$-adic matrices. This begged the question of whether the multiplicative Brownian motion on $\GL_N(\C)$ and multiplicative Dyson Brownian motion have analogues in the $p$-adic setting. To better phrase this question we must go beyond explicit constructions and understand what characterizes multiplicative Brownian motion. %. This work concerns the $p$-adic analogue of the multiplicative Brownian motion. 

\subsection{Invariant stochastic processes on matrices} For multiplicative Brownian motion $\sY^{(N)}(T)$, it is natural to consider its \emph{multiplicative increments} $\sY^{(N)}(t_{i})\sY^{(N)}(t_{i-1})^{-1}$ for a series of times $t_1 < t_2 < \ldots < t_k$, as the value of the multiplicative Brownian motion at a given $t_i$ is the product of the corresponding increments. These increments $\sY^{(N)}(T+s)\sY^{(N)}(T)^{-1}$ satisfy
\begin{enumerate}[label=(\roman*)]
\item Independence: $\sY^{(N)}(T+s)\sY^{(N)}(T)^{-1}$ is independent of the trajectory $\sY^{(N)}(\tau), 0 \leq \tau \leq T$,
\item Stationarity: $\sY^{(N)}(T+s)\sY^{(N)}(T)^{-1} = \sY^{(N)}(s)\sY^{(N)}(0)^{-1} $ in distribution for any $T \geq 0$, and
\item Isotropy with respect to $\U(N)$: For any $U \in \U(N)$, $\sY^{(N)}(T+s)\sY^{(N)}(T)^{-1} = U\sY^{(N)}(T+s)\sY^{(N)}(T)^{-1}U^{-1}$ in distribution. Equivalently, $\Pr(\sY^{(N)}(T+s) \in S| \sY(T) = x) = \Pr(\sY^{(N)}(T+s) \in US| \sY(T) = Ux)$ for any $U \in \U(N)$, $x \in \GL_N(\C)$ and measurable $S \subset \GL_N(\C)$. \label{item:isotropy}
\end{enumerate}

The first two are familiar from the theory of Brownian motion on $\R$, while a version of the latter with respect to the rotation group $O(N)$ may be seen as soon as one considers Brownian motion on $\R^N$. The fixed-time marginals of any process satisfying the above and starting at the Haar measure must be infinitely-divisible measures invariant with respect to the action of $\U(N)$ on the left and right, which were classified by the generalized Lévy-Khintchine theorem of Gangolli \cite{gangolli1964isotropic} (see also earlier work of Hunt \cite{hunt1956semi}). Later, Gangolli \cite{gangolli1965sample} explicitly constructed stochastic processes with these fixed-time marginals, finding them to be mixtures of multiplicative Brownian motion and Poisson jump processes, as with the classical Lévy-Khintchine theorem on $\mathbb{R}$. Of these, only the multiplicative Brownian motion\footnote{This uniqueness actually applies after restricting to the subgroup $\SL_N(\C)$. In the case of $\GL_N(\C)$ there is an additional multiplicative Brownian motion on $\R_+$ corresponding to the determinant, leading to a two-parameter family of processes with continuous sample paths. See Jones-O'Connell \cite[p108]{jones2006weyl} for discussion.} has continuous sample paths. We note that strictly speaking, the uniqueness statement applies only to the infinitely divisible measures which are the single-time marginals of the process, and we are not aware of a uniqueness statement at the process level analogous to the characterization of Brownian motion on $\R$.

One might optimistically hope for a similar classification in the $p$-adic case, and hope that the `right' analogue of multiplicative Brownian motion yields an elegant stochastic process on singular numbers similar to the above multiplicative Dyson Brownian motion. Both hopes will turn out to be well-founded.

\subsection{Classifying invariant processes} Now we turn to the question of Markov processes on $\mathrm{GL}_{N}(\mathbb{Q}_{p})$ with stationary, independent, $\GL_N(\Z_p)$-isotropic increments, where $\GL_N(\Z_p)$-isotropy is defined by \ref{item:isotropy} above with $\GL_N(\C)$ replaced by $\GL_N(\Q_p)$ and $\U(N)$ replaced by $\GL_N(\Z_p)$. The following definition gives a wide class of processes which are easily seen to have these properties.

\begin{defi}\label{def:the_processes}
Let $N \in \Z_{\geq 1}$, let $M$ be any probability measure on $\Sig_N$, and let $c \in \R_{\geq 0}$. Then we define the process $Y^{(N,M,c)}(\tau), \tau \in \R_{\geq 0}$ on $\GL_N(\Q_p)$ by 
\begin{equation}
Y^{(N,M,c)}(\tau) := U_{P(\tau)} \diag(p^{\nu^{(P(\tau))}_1},\ldots,p^{\nu^{(P(\tau))}_N}) V_{P(\tau)}  \cdots U_1 \diag(p^{\nu^{(1)}_1},\ldots,p^{\nu^{(1)}_N}) V_1 
\end{equation}
where $P(\tau)$ is a Poisson process on $\Z_{\geq 0}$ with rate $c$, and $\nu^{(i)} \sim M$ and $U_i,V_i \sim M_{Haar}(\GL_N(\Z_p))$ are iid.
\end{defi}

Our first result, \Cref{thm:classify_processes_intro}, says that at the level of singular numbers \Cref{def:the_processes} is the only example.

\begin{thm}\label{thm:classify_processes_intro}
 Let $N \in \mathbb{Z}_{\geq 1}$ and let $X(\tau), \tau \in \mathbb{R}_{\geq 0}$ be a Markov process on $\mathrm{GL}_{N}(\mathbb{Q}_{p})$ started at the identity with stationary, independent, $\GL_N(\Z_p)$-isotropic increments. Then there exists a constant $c \in \mathbb{R}_{\geq 0}$ and a probability measure $M_X$ on $\operatorname{Sig}_{N}$ such that 
\begin{equation}
\SN(X(\tau)) = \SN(Y^{(N,M_X,c)}(\tau)) \quad \quad \quad \quad \text{ in multi-time distribution.}
\end{equation}
\end{thm}

We deduce \Cref{thm:classify_processes_intro} from a later result stated as \Cref{thm:classify_processes_cts_G/K}, which works at the level of the homogeneous space $\mathrm{GL}_{N}(\mathbb{Q}_{p}) / \mathrm{GL}_{N}(\mathbb{Z}_{p})$ analogously to \cite{gangolli1964isotropic}. We note that because the dynamics of \Cref{def:the_processes} commutes with the right action of $\GL_N(\Z_p)$, it projects to Markov dynamics on $\mathrm{GL}_{N}(\mathbb{Q}_{p}) / \mathrm{GL}_{N}(\mathbb{Z}_{p})$. Even at this level, unlike the homogeneous space $\mathrm{GL}_{N}(\mathbb{C}) / \U(N)$, the space $\mathrm{GL}_{N}(\mathbb{Q}_{p}) / \mathrm{GL}_{N}(\mathbb{Z}_{p})$ is countable and naturally carries the discrete topology. Hence one cannot expect an analogue of the continuous multiplicative Brownian motion, but it turns out that the same Poisson jump/matrix product processes appear as in the complex case---the processes $Y^{(N,M,c)}$ provide examples of these.

Given that no process on $\GL_N(\Q_p)/\GL_N(\Z_p)$ with continuous paths exists, one may at least ask for a Poisson jump process with the smallest or simplest jumps. Such a process should have $M_X$ supported on $\delta_{(0[N])}$ and the smallest nontrivial signature $\delta_{(1,0[N-1])}$, where here and below we use the notation $a[k]$ in signatures to denote the integer $a$ repeated $k$ times, so $(0[N])$ is the $N$-tuple with all entries $0$ and $(1,0[N-1])$ is the $N$-tuple with all entries $0$ except for the first $1$. Note that one might equally well replace $\delta_{(1,0[N-1])}$ by $\delta_{(0[N-1],-1)}$, but this is related to the previous case by taking inverse matrices and reversing the left and right actions, so there is no loss in our choice. Because the singular numbers of $Y^{(N,M_X,c)}$ do not change at the jumps where $\nu^{(i)} = (0[N])$, as far as the singular numbers are concerned one may take $M_X = \delta_{(1,0[N-1])}$, up to changing the Poisson rate constant $c$. We see next that the singular numbers of this process have an elegant description, which is our main result. 

% \fix{think about how to integrate with rmk/move around}
% Let us be clear that while we believe \Cref{thm:classify_processes_intro} and \Cref{thm:dbm_poisson_intro}, together with the above discussion, give a satisfactory answer to the question posed in the title, we do \emph{not} answer the stronger question
% \begin{equation*}
% \text{\textit{What is the analogue of multiplicative Brownian motion for $\GL_N(\Q_p)$?}}
% \end{equation*}
% We have only shown that the singular numbers of such a process should agree with those of $Y^{(N,\delta_{(1,0[N-1])},c)}$, but we have made no uniqueness statement at the level of a process on $\GL_N(\Q_p)$. The process $Y^{(N,\delta_{(1,0[N-1])},c)}$ is quite natural, but since $\GL_N(\Q_p)$ is not a discrete space, a reasonable analogue of multiplicative Brownian motion should probably not be a Poisson jump process. We return to this question in \Cref{rmk:other_works_p-adic_bm}.

\subsection{$p$-adic Dyson Brownian motion}
\begin{defi}\label{def:poisson_walks}
Fix $N \in \mathbb{N}$ and a parameter $t \in (0,1)$. We define the stochastic process $\Pois^{(N)}(\tau)=\left(\Pois_{1}^{(N)}(\tau), \ldots, \Pois_{N}^{(N)}(\tau)\right)$ on $\operatorname{Sig}_{N}$ to be a Poisson process on $\Z^N$ with independent jump rates in each coordinate direction given by $(t,t^2,\ldots,t^N)$, which reflects off the boundary of $\W_N$ and has initial condition $\Pois^{(N)}(0) = (0[N])$. More explicitly, for each $1 \leq i \leq N$, $ \Pois_{i}^{(N)}$ has an exponential clock of rate $t^{i}$, and when the clock associated to $\Pois_{i}^{(N)}$ rings, $\Pois_{i}^{(N)}$ increases by $1$ if the resulting $n$-tuple is still weakly decreasing. If not, then $\Pois_{i-d}^{(N)}$ increases by $1$ instead and $\Pois_{i}^{(N)}$ remains the same, where $d \geq 0$ is the smallest nonnegative integer so that the resulting tuple is weakly decreasing.  
\end{defi}

\begin{figure}[htbp]
\begin{center}
\includegraphics[scale=1]{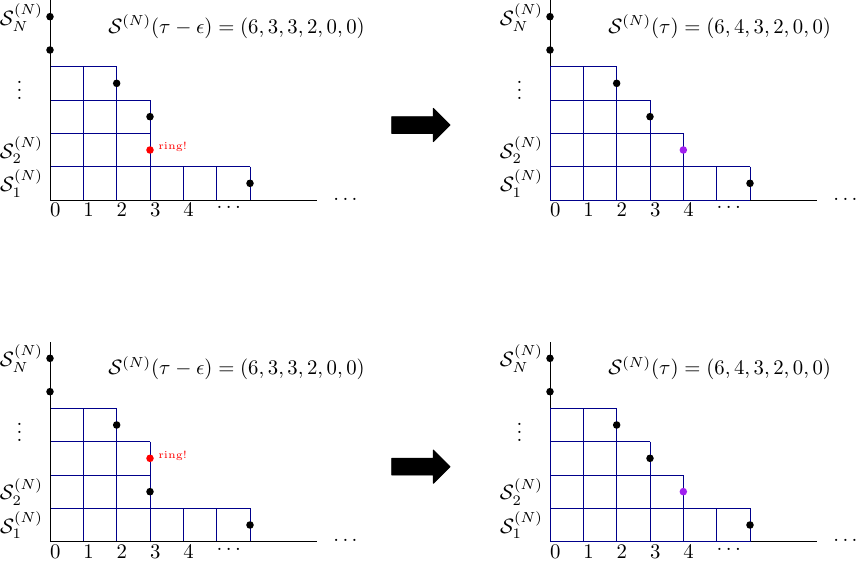}
\end{center}
\caption{A depiction of the process $\Pois^{(N)}(\tau) = (\Pois_1^{(N)}(\tau),\ldots,\Pois_N^{(N)}(\tau))$ as a Young diagram in English notation when $N=6$. If the clock associated to $\Pois^{(N)}_2 = 2$ rings at time $\tau$ and the process was previously in state $(6,3,3,2,0,0)$ (i.e. $\Pois^{(N)}(\tau-\eps) = (6,3,3,2,0,0)$ for all sufficiently small $\eps > 0$), then $\Pois^{(N)}_2 $ increases by $1$ and so $\Pois^{(N)}(\tau) = (6,4,3,2,0,0)$ (above). If instead the clock associated to $\Pois^{(N)}_3 = 2$ rings at time $\tau$ and the process was previously in state $(6,3,3,2,0,0)$, then $\Pois^{(N)}_2 $ increases instead by $1$ and so still $\Pois^{(N)}(\tau) = (6,4,3,2,0,0)$ (below), demonstrating the reflection condition.}\label{fig:poisson_walks_horiz}
\end{figure}

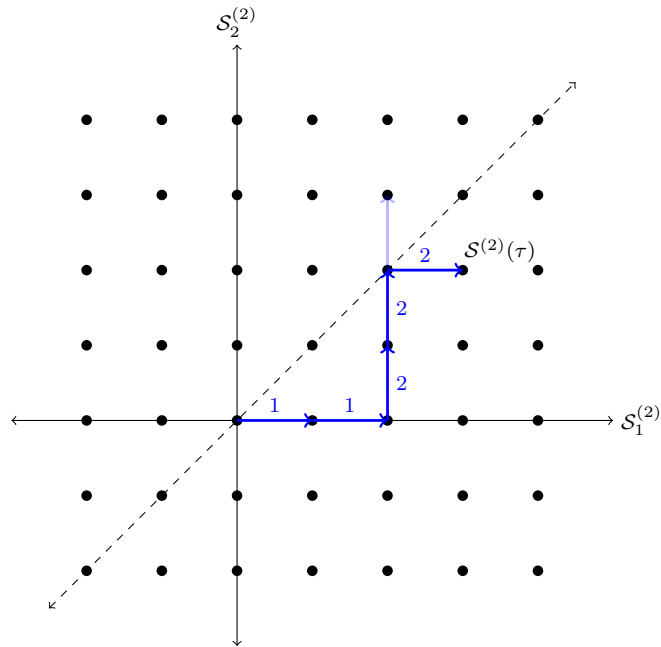
\begin{figure}[htbp]
\begin{center}
\begin{tikzpicture}[scale=1]

  % Axes
  \draw[<->] (-3,0) -- (5,0) node[right] {$\Pois^{(2)}_1$};
  \draw[<->] (0,-3) -- (0,5) node[above] {$\Pois^{(2)}_2$};

  % Lattice points
  \foreach \x in {-2,...,4}
    \foreach \y in {-2,...,4}
      \fill (\x,\y) circle (2pt);

  % Line x=y
  \draw[<->,dashed] (-2.5,-2.5) -- (4.5,4.5);

  \draw[->,line width = 1.05 pt,blue] (0,0) -- (1,0)  node[midway, above] {$1$};
  \draw[->,line width = 1.05 pt,blue] (1,0) -- (2,0) node[midway, above] {$1$};
  \draw[->,line width = 1.05 pt,blue] (2,0) -- (2,1) node[midway, right] {$2$};
  \draw[->,line width = 1.05 pt,blue] (2,1) -- (2,2) node[midway, right] {$2$};
  \draw[->,line width = 1.05 pt,blue] (2,2) -- (3,2) node[midway, above] {$2$};
  \draw[->,line width = 1.05 pt,blue,opacity=.3] (2,2) -- (2,3);

  \node[above] at (3.5,2) {$\Pois^{(2)}(\tau)$};

\end{tikzpicture}
\end{center}
\caption{The reflection condition of \Cref{def:poisson_walks} in the case $n=2$: here $(\Pois^{(2)}_1(\tau),\Pois^{(2)}_2(\tau))$ is portrayed as an up-right walk in the $x-y$ plane lying below the line $y=x$, and each jump is labeled by which clock rings. In the final jump, the second clock rings, but due to the reflection condition, $\Pois^{(2)}_2$ does not increase---the result of such an increase is shown as an opaque arrow---but rather $\Pois^{(2)}_1$ increases instead.}\label{fig:n=2_reflection}
\end{figure}

\begin{thm}\label{thm:dbm_poisson_intro}
Let $X^{(N)}(\tau) := Y^{(N,\delta_{(1,0[N-1])},1)}(\tau)$ in the notation of \Cref{def:the_processes}. Then
\begin{equation}
\SN\left(X^{(N)}(\tau)\right)=\Pois^{(N)}\left(\left(\frac{1}{t} \frac{1-t}{1-t^{N}}\right) \tau\right)
\end{equation}
in multi-time distribution, where we take the parameter $t$ in $\Pois^{(N)}$ to be $1/p$.
\end{thm}

$\Pois^{(N)}$ and the multiplicative Dyson Brownian motion $(\log \y_1(\tau),\ldots,\log \y_N(\tau))$ both have the singular numbers/log singular values evolving independently, subject to a constraint which keeps them in decreasing order. In the former case the independent parts are Poisson processes with rates in geometric progression, and the constraint takes the form of reflection off the walls of the Weyl chamber $\W_N$. In the latter case they are independent Brownian motions with drifts in arithmetic progression, kept ordered by conditioning to avoid the walls of $\W_N$ for all time. These two constraints result in very different dynamics: when $\Pois^{(N)}$ lies in the interior of $\W_N$, the coordinates evolve independently, while the conditioning of multiplicative Dyson Brownian motion to avoid the walls of $\W_N$ for all time means that the $\log \y_i$ repel one another even when far away. 

Our motivation to highlight $\Pois^{(N)}$ in this paper also comes from a related work \cite{van2023reflecting} where its reflected Poisson dynamics are shown to govern the local dynamics of singular numbers of $p$-adic matrix products under weak universality hypotheses. Similar results relating matrix products to multiplicative Dyson Brownian motion in the complex setting have been shown by Ahn \cite{ahn2022extremal} (see also \cite{ahn2023lyapunov}). See \cite{van2023reflecting} for more detail on the analogy between these results in the complex and $p$-adic settings.

The first appearance of $\Pois^{(N)}$ we are aware of in the literature is in work of Borodin announced in \cite{borodin1995limit}, with detailed proofs appearing later in \cite{borodin1999lln}. That work considered a process on $\Sig_\infty$ which was essentially the $N=\infty$ version of $\Pois^{(N)}$ (see \Cref{def:poisson_walks_generator}), in a discrete-time form where the jumps occur at each time step rather than in continuous time according to an exponential clock. This process was shown to describe the evolution of the Jordan form of growing strictly upper-triangular matrices over $\F_q$, with matrix size playing the role of time. This time-discretized version of $\Pois^{(\infty)}$ was defined in \cite{borodin1995limit} by explicit transition probabilities, without the connection to reflecting random walks, see \cite[Theorem 2.3]{borodin1999lln} (which reinterpreted an earlier equivalent computation by Kirillov \cite[Theorem 2.3.2]{kirillov1995variations}). We note that the arguments in the present paper apply also with $\Z_p$ replaced by $\F_q\llbracket T\rrbracket$ (and $t=1/p$ by $t=1/q$), and the action of $T$ on a finite-dimensional $\F_q\llbracket T\rrbracket$-module is nilpotent and hence gives an upper-triangular matrix in an appropriate basis. Hence the setting of these works is very similar to the present one, and it seems possible that one could derive \cite[Theorem 2.3]{borodin1999lln} from an $N \to \infty$ limit of our result. 

Another interesting feature is that $\Pois^{(N)}$ is a so-called \emph{Hall-Littlewood process}, which was shown in \cite[Proposition 3.4]{van2022q} in the context of interacting particle systems, before the random matrix context in the present work was understood. In addition to motivating $\Pois^{(N)}$ as a natural object, this allows us to analyze its asymptotics using the Macdonald process tools of Borodin-Corwin \cite{borodin2014macdonald}, which we do in a separate work \cite{van2023local}. We mention also that Macdonald processes (in the Schur and Heckman-Opdam limits) have also been useful in analyzing nonintersecting path models, see e.g. Johansson \cite[Section 3]{johansson2002non}, and that the relation of Hall-Littlewood polynomials and the setting of \cite{borodin1995limit,borodin1999lln} is stated more explicitly in the recent work of Cuenca-Olshanski \cite{cuenca2022infinite}.

\begin{rmk}\label{rmk:reflected_rw}
Random walks on Weyl chambers conditioned to never intersect are the subject of an extensive literature with connections to representation theory, total positivity, and other parts of combinatorics, as well as random matrices. See Biane \cite{biane1991quantum,biane1992minuscule}, Grabiner \cite{grabiner1999brownian}, Baryshnikov \cite{baryshnikov2001gues}, Bougerol-Jeulin \cite{bougerol2002paths}, O'Connell-Yor \cite{o2002representation}, Biane-Bougerol-O'Connell \cite{bougerol2005littelmann}, and the references therein. However, we are not aware of corresponding work for reflected (rather than conditioned) random walks on a positive Weyl chamber, corresponding to \ref{item:rsk} in the Preface, and believe it is worth understanding whether the combinatorics is similarly rich. To clarify a potential point of confusion, let us note that reflections across the walls of the Weyl chamber appear across the works which treat conditioned random walks, in analogues of the classical reflection principle for Brownian motion following Gessel-Zeilberger \cite{gessel1992random}; however, the random walks themselves are not reflected at the boundary, but conditioned to avoid it.
\end{rmk}

\begin{rmk}\label{rmk:bt_buildings}
A related body of literature deals with random walks on Bruhat-Tits buildings, of which $\SL_N(\Q_p)/\SL_N(\Z_p)$ is the type $\tilde{A}$ case, see e.g. Parkinson \cite{parkinson2017buildings} and the references therein. These typically treat random walks satisfying a stronger notion of isotropy than ours: theirs in our context would be the assumption that
\begin{equation}\label{eq:strong_isotropy}
\Pr(X(\tau+s) = y| X(\tau) = x) = \Pr(X(\tau+s) = Uy| X(\tau) = Vx) 
\end{equation}
for any fixed $U,V \in \GL_N(\Z_p)$, while ours only requires
\begin{equation}\label{eq:weak_isotropy}
\Pr(X(\tau+s) = y| X(\tau) = x) = \Pr(X(\tau+s) = Uy| X(\tau) = Ux). 
\end{equation}
It is not hard to show by slight modifications of our arguments that the only processes satisfying the strong isotropy condition \eqref{eq:strong_isotropy} and stationary independent increments are of the form $Y^{(N,M,c)}(\tau)$, and indeed this is remarked in the discrete-time setting in Parkinson \cite[p381]{parkinson2007isotropic}. 

However, multiplicative Brownian motion on $\GL_N(\C)/\U(N)$ does \emph{not} satisfy the strong isotropy condition \eqref{eq:strong_isotropy}; indeed, this condition in continuous time precludes continuous sample paths. This is our reason for taking the weaker condition \eqref{eq:weak_isotropy}, which multiplicative Brownian motion does satisfy, even though the resulting constraints on the processes in \Cref{thm:classify_processes_cts_G/K} are weaker than one obtains with \eqref{eq:strong_isotropy}. 
%The continuous-time version \Cref{thm:classify_processes_intro} does not go beyond these existing works in any meaningful way (though we could not find an exact reference), and our reason for proving it here is simply to provide context for our main result, \Cref{thm:dbm_poisson_intro}. We note also that the continuous-time `Poissonized' process $\Pois^{(N)}$ is much easier than the corresponding discrete-time version to study asymptotically due to the aforementioned connection to Hall-Littlewood processes, see \cite{van2023local}. This is very analogous to the Poissonized Plancherel measure on symmetric groups, which is easier to study than the usual Plancherel measure but may be `de-Poissonized' to obtain results on the later, see e.g. \cite{borodin2000asymptotics}. \fix{probably move to other paper, and rewrite this remark}
\end{rmk}

\begin{rmk}\label{rmk:additive_case}
    The reason we study $p$-adic analogues of multiplicative Dyson Brownian motion, rather than the related additive Dyson Brownian motion on Hermitian matrices, is as follows. In the complex case, it is natural to view the projection of multiplicative Brownian motion to the homogeneous space $\GL_N(\C)/\U(N)$ as we have explained. The tangent space to $\GL_N(\C)/\U(N)$ is $\mathfrak{gl}_N/\mathfrak{u}_N$, which as a vector space is just $ \{\bi A: A \in \Mat_N(\C) \text{  Hermitian}\}$, the state space of additive Hermitian Brownian motion (up to multiplication by $\bi$). This viewpoint and relation between matrix sum and matrix product processes are nicely explained by Klyachko \cite{klyachko2000random}. Note that at the level of spectra, the additive Dyson Brownian motion is also a linearization of the multiplicative Dyson Brownian motion, since the limit of nonintersecting Brownian motion with drift at small times (with space rescaled appropriately) is just nonintersecting Brownian motion with no drift. 
    
    The analogies we are aware of between the complex and $p$-adic settings, explained for instance in \cite{van2020limits}, are at the level of the linear algebraic groups rather than their tangent spaces. Hence this was the natural setting for the investigations here. We are not sure if there is a reasonable $p$-adic analogue of the additive Dyson Brownian motion on the tangent space, though this is certainly an interesting question.
\end{rmk}

%\begin{rmk}\label{rmk:other_works_p-adic_bm}
\subsection{$p$-adic multiplicative Brownian motion?}
We believe \Cref{thm:classify_processes_intro} and the discussion directly after it, together with \Cref{thm:dbm_poisson_intro}, give a satisfactory answer to the question posed in the title. However, let us be clear that we have \emph{not} answered the stronger question
\begin{equation*}
\text{\textit{What is the analogue for $\GL_N(\Q_p)$ of multiplicative Brownian motion on $\GL_N(\C)$?}}
\end{equation*}
We have only shown that the singular numbers of such a process must agree with those of the Poissonized matrix product process $Y^{(N,\delta_{(1,0[N-1])},c)}$ in multi-time distribution, but we have made no uniqueness statement at the level of a process on $\GL_N(\Q_p)$. The process $Y^{(N,\delta_{(1,0[N-1])},c)}$ is quite natural, but there may certainly be a more natural one. Returning to \Cref{rmk:bt_buildings}, we expect that such a process may not satisfy the strong isotropy condition \eqref{eq:strong_isotropy}. It additionally seems plausible that a good analogue of multiplicative Brownian motion should not wait at any state for a nonzero amount of time, which $Y^{(N,\delta_{(1,0[N-1])},c)}$ does.  %but the fact that it waits at each state for a nonzero amount of time seems
%In \Cref{thm:classify_processes_intro} we have shown that at the level of the singular numbers, all continuous-time Markov processes on $\GL_N(\Q_p)$ with appropriate invariance properties are simply Poissonized matrix product processes. However, we have cheated the reader out of an actual construction of a continuous-time process on $\GL_N(\Q_p)$ which is not a Markov jump process, by simply showing that the singular numbers of such a process would have to agree with those of a Poissonized matrix product process (\Cref{def:our_process}) and considering the latter. It would be interesting to see if there are canonical, more naturally continuous-time constructions on $\GL_N(\Q_p)$, i.e. natural processes on $\GL_N(\Q_p)$ which do not wait at any state for a nonzero amount of time, which yield the same process at the level of singular numbers. %Of course, because $\Q_p$ is totally disconnected, no process with continuous sample paths will exist. 

We believe that finding analogues for $\GL_N(\Q_p)$ of multiplicative Brownian motion is an interesting problem for the future, and hence record some thoughts for any reader wishing to try this. In the simplest case $N=1$, the above discussion concerns processes in continuous ($\R$-valued) time on the group $\Q_p^\times$, for which existing literature on $p$-adic Brownian motions such as Albeverio-Karwowski \cite{albeverio1991diffusion,albeverio1994random} or Evans \cite{evans1989local} (which studies more general totally disconnected abelian groups) likely provides a natural route to answering the question above when $N=1$. However, we are not aware of works concerning stochastic processes on nonabelian $p$-adic groups. 

To prevent confusion for one who wishes to begin reading the primary sources on $p$-adic Brownian motions, it is worth mentioning that many previous works referring to $p$-adic Brownian motions such as Evans \cite{evans1993local,evans2001local} and Bikulov-Volovich \cite{bikulov1997p} are instead studying a process where the time parameter lives in $\Q_p$ rather than $\R$, leading to a different object which has no \emph{a priori} relation to our setting.

% Let us be clear that while we believe \Cref{thm:classify_processes_intro} and \Cref{thm:dbm_poisson_intro}, together with the above discussion, give a satisfactory answer to the question posed in the title, we do \emph{not} answer the stronger question
% \begin{equation*}
% \text{\textit{What is the analogue of multiplicative Brownian motion for $\GL_N(\Q_p)$?}}
% \end{equation*}
% We have only shown that the singular numbers of such a process should agree with those of $Y^{(N,\delta_{(1,0[N-1])},c)}$, but we have made no uniqueness statement at the level of a process on $\GL_N(\Q_p)$. The process $Y^{(N,\delta_{(1,0[N-1])},c)}$ is quite natural, but since $\GL_N(\Q_p)$ is not a discrete space, a reasonable analogue of multiplicative Brownian motion should probably not be a Poisson jump process. We return to this question in \Cref{rmk:other_works_p-adic_bm}.

\subsection{Notation} Fix a prime $p$ throughout. We write formulas for a real number depending on $p$ in terms of $t=1/p$, and use $t$ interchangeably for $1/p$ and for the parameter in \Cref{def:poisson_walks}, as the latter will be set to $1/p$ anyway when $\Pois^{(N)}$ appears in random matrix theory. The letter $p$ is reserved for the element of $\Q_p$ or $\Z/p^n\Z$ as opposed to the real number.

\subsection{Outline} In \Cref{sec:p-adic} we give basic background and some necessary preliminary results on $\Q_p$ and $p$-adic random matrices. In \Cref{sec:classify_processes} we prove results on processes on $\GL_N(\Q_p)$ and $\GL_N(\Q_p)/\GL_N(\Z_p)$ in discrete and continuous time, and use the latter to derive \Cref{thm:classify_processes_intro}. In \Cref{sec:poisson_walks} we prove \Cref{thm:dbm_poisson_intro} by linear-algebraic computations.

\textbf{Acknowledgments.} I am very grateful to Alexei Borodin for invaluable discussions, comments on the draft, and clarifying the history of his and Kirillov's work. I also wish to thank Andrew Ahn for discussions on (complex) multiplicative Brownian motion, Theo Assiotis for encouraging questions about $p$-adic analogues of Dyson Brownian motion, and the two referees for constructive comments which improved the exposition. This work was partially supported by an NSF Graduate Research Fellowship under grant \#1745302.

\section{$p$-adic matrix preliminaries} \label{sec:p-adic}

We begin with a few paragraphs of background which are essentially quoted from \cite{van2020limits}, and are a condensed version of the exposition in Evans \cite[Section 2]{evans2002elementary}, to which we refer any reader desiring a more detailed introduction to $p$-adic numbers geared toward a probabilistic viewpoint. Fix a prime $p$. Any nonzero rational number $r \in \Q^\times$ may be written as $r=p^k (a/b)$ with $k \in \Z$ and $a,b$ coprime to $p$. Define $|\cdot|: \Q \to \R$ by setting $|r|_p = p^{-k}$ for $r$ as before, and $|0|_p=0$. Then $|\cdot|_p$ defines a norm on $\Q$ and $d_p(x,y) :=|x-y|_p$ defines a metric. We additionally define $\val_p(r)=k$ for $r$ as above and $\val_p(0) = \infty$, so $|r|_p = p^{-\val_p(r)}$. We define the \emph{field of $p$-adic numbers} $\Q_p$ to be the completion of $\Q$ with respect to this metric, and the \emph{$p$-adic integers} $\Z_p$ to be the unit ball $\{x \in \Q_p : |x|_p \leq 1\}$. It is not hard to check that $\Z_p$ is a subring of $\Q_p$. We remark that $\Z_p$ may be alternatively defined as the inverse limit of the system $\ldots \to \Z/p^{n+1}\Z \to \Z/p^n \Z \to \cdots \to \Z/p\Z \to 0$, and that $\Z$ naturally includes into $\Z_p$. 

$\Q_p$ is noncompact but is equipped with a left- and right-invariant (additive) Haar measure; this measure is unique if we normalize so that the compact subgroup $\Z_p$ has measure $1$. The restriction of this measure to $\Z_p$ is the unique Haar probability measure on $\Z_p$, and is explicitly characterized by the fact that its pushforward under any map $r_n:\Z_p \to \Z/p^n\Z$ is the uniform probability measure. For concreteness, it is often useful to view elements of $\Z_p$ as `power series in $p$' $a_0 + a_1 p + a_2 p^2 + \ldots$, with $a_i \in \{0,\ldots,p-1\}$; clearly these specify a coherent sequence of elements of $\Z/p^n\Z$ for each $n$. The Haar probability measure then has the alternate explicit description that each $a_i$ is iid uniformly random from $\{0,\ldots,p-1\}$. Additionally, $\Q_p$ is isomorphic to the ring of Laurent series in $p$, defined in exactly the same way.

Similarly, $\GL_N(\Q_p)$ has a unique left- and right-invariant measure for which the total mass of the compact subgroup $\GL_N(\Z_p)$ is $1$. The restriction of this measure to $\GL_N(\Z_p)$, which we denote by $M_{Haar}(\GL_N(\Z_p))$, pushes forward to $\GL_N(\Z/p^n\Z)$ and is the uniform measures on the finite groups $\GL_N(\Z/p^n\Z)$. This gives an alternative characterization of the Haar measure. The Haar measure on $\GL_N(\Z_p)$ also has a useful explicit characterization.

\begin{prop}\label{thm:haar_sampling}
Let $A \in \Mat_N(\Z_p)$ be a random matrix with distribution given as follows: sample its columns $v_N,v_{N-1},\ldots,v_1$ from right to left, where the conditional distribution of $v_i$ given $v_{i+1},\ldots,v_N$ is that of a random column vector with additive Haar distribution conditioned on the event 
\begin{equation}\label{eq:notinspan}
v_i \pmod{p} \not \in \Span(v_{i+1} \pmod{p},\ldots,v_N \pmod{p}) \subset \F_p^N,
\end{equation}
where in the case $i=N$ we take the span in \eqref{eq:notinspan} to be the $0$ subspace. Then $A$ is distributed by the Haar measure on $\GL_N(\Z_p)$.
\end{prop}
\begin{proof}
Because $\GL_N(\Z_p)$ is compact, it suffices to show the above is a left Haar measure, i.e. for any $B \in \GL_N(\Z_p)$ we must show $BA=A$ in distribution. We show $(v_{N-j},\ldots,v_N) = (Bv_{N-j},\ldots,Bv_N)$ in distribution for any $j$ by induction, which suffices. For the base case, recall (see e.g. \cite{evans2001local}) that additive Haar measure on $\Z_p^N$ is invariant under $\GL_N(\Z_p)$, and $Bv_N \equiv 0 \pmod{p}$ if and only if $v_N \equiv 0 \pmod{p}$, hence $Bv_N = v_N$ in distribution. For the inductive step, we have that $Bv_{N-j+1},\ldots,Bv_N$ satisfy \eqref{eq:notinspan} with $i=N-j+1,\ldots,N$ if and only $v_{N-j+1},\ldots,v_N$ do. Furthermore, for any $(w_{N-j+1},\ldots,w_N)$ in the support of $\Law(v_{N-j+1},\ldots,v_N)$, we have
\begin{equation}
\Law(Bv_{N-j}| v_{N-i} = w_{N-i} \text{ for all }0 \leq i < j) = \Law(v_{N-j}| v_{N-i} = Bw_{N-i} \text{ for all }0 \leq i < j).
\end{equation}
It follows by the inductive hypothesis that 
\begin{equation}
\Law(v_{N-j},\ldots,v_N) = \Law(Bv_{N-j},\ldots,Bv_N),
\end{equation}
completing the proof.
\end{proof}

The following standard result is sometimes known either as Smith normal form or Cartan decomposition.

\begin{prop}\label{prop:smith}
Let $n \leq m$. For any nonsingular $A \in \Mat_{n \times m}(\Q_p)$, there exist $U \in \GL_n(\Z_p), V \in \GL_m(\Z_p)$ such that $UAV = \diag_{n \times m}(p^{\la_1},\ldots,p^{\la_n})$ where $\la \in \Sig_n$. Furthermore, there is a unique such $n$-tuple $\la$. The parts $\la_i$ are known as the \emph{singular numbers} of $A$, and we denote them by $\SN(A) = \la = (\la_1,\ldots,\la_n)$.
\end{prop}

We will often write $\diag_{n \times N}(p^\la)$ for $\diag_{n \times N}(p^{\la_1},\ldots,p^{\la_n})$, and also omit the dimensions $n \times N$ when they are clear from context. We note also that for any $\la \in \Sig_N$, the orbit $\GL_N(\Z_p) \diag(p^\la) \GL_N(\Z_p)$ is compact. The restriction of the Haar measure on $\GL_N(\Q_p)$ to such a double coset, normalized to be a probability measure, is the unique $\GL_N(\Z_p) \times \GL_N(\Z_p)$-invariant probability measure on $\GL_N(\Q_p)$ with singular numbers $\la$, and all $\GL_N(\Z_p) \times \GL_N(\Z_p)$-probability measures and convex combinations of these for different $\la$. These measures may be equivalently described as $U \diag(p^{\la_1},\ldots,p^{\la_N}) V$ where $U,V$ are independently distributed by the Haar probability measure on $\GL_N(\Z_p)$. More generally, if $n \leq m$ and $U \in \GL_n(\Z_p), V \in \GL_m(\Z_p)$ are Haar distributed and $\mu \in \Sig_n$, then $U \diag_{n \times m}(p^\mu) V$ is invariant under $\GL_n(\Z_p) \times \GL_m(\Z_p)$ acting on the left and right, and is the unique such bi-invariant measure with singular numbers given by $\mu$. %These bi-invariant measures on rectangular matrices are the ones which appear in \Cref{thm:exact_hl_results_in_p-adic_rmt}.

Similarly to eigenvalues and singular values, singular numbers have a variational characterization. We first recall the version for singular values, one version of which states that for $A \in \Mat_{n \times m}(\C)$ (assume without loss of generality $n \leq m$) with singular values $a_1 \geq \ldots \geq a_n$,
\begin{equation}\label{eq:sv_minmax}
\prod_{i=1}^k a_i = \sup_{\substack{V \subset \C^m: \dim(V) = k \\ W \subset \C^n: \dim(W) = k}} |\det(\Proj_W \circ A|_V)|
\end{equation}
where $\Proj$ is the orthogonal projection and $A|_V$ is the restriction of the linear operator $A$ to the subspace $V$. \eqref{eq:sv_minmax} holds because the right hand side is unchanged by multiplying $A$ by unitary matrices, hence $A$ may be taken to be diagonal with singular values on the diagonal by singular value decomposition, at which point the result is easy to see. For a slightly different version which picks out the $k\tth$ largest singular value rather than the product of the $k$ largest, see Fulton \cite[Section 5]{fulton2000eigenvalues}.

For $p$-adic matrices, we state the result slightly differently to avoid referring to orthogonal projection, the reason being that unlike $\U(n)$, $\GL_n(\Z_p)$ does not preserve a reasonable inner product, only the norm. We give two versions, \Cref{thm:minmax} and \Cref{thm:submatrices_suffice}, the former of which is more transparent to prove, and the latter of which is more useful in practice when dealing with matrices with explicit entries.

\begin{prop}\label{thm:minmax}
Let $1 \leq n \leq m$ be integers and $A \in \Mat_{n \times m}(\Q_p)$ with $\SN(A) = (\la_1,\ldots,\la_n)$. Then for any $1 \leq k \leq n$,
\begin{align}\label{eq:raleigh}
\la_n+\ldots+\la_{n-k+1} &= \inf_{\substack{P:\Q_p^n \to \Q_p^n \text{ rank $k$ projection}\\ W \subset \Q_p^m: \dim W = k}} \val_p(\det(PA|_W)). 
\end{align}
\end{prop}
\begin{proof}
If $U_1 \in \GL_n(\Z_p),U_2 \in \GL_m(\Z_p)$, then for any a rank $k$ projection $P$ the matrix $U_1PU_1^{-1}$ is also a rank $k$ projection, and similarly for any $W$ as above $U_2 W$ is also a dimension $k$ subspace. Hence
\begin{equation}
\inf_{\substack{P:\Q_p^n \to \Q_p^n \text{ rank $k$ projection}\\ W \subset \Q_p^m: \dim W = k}} \val_p(\det(PA|_W)) = \inf_{\substack{P:\Q_p^n \to \Q_p^n \text{ rank $k$ projection}\\ W \subset \Q_p^m: \dim W = k}} \val_p(\det(P(U_1AU_2)|_W)).
\end{equation}
By Smith normal form we may choose $U_1,U_2$ so that $U_1AU_2 = \diag(p^{\la_1},\ldots,p^{\la_n})$, hence
\begin{equation}
\text{RHS\eqref{eq:raleigh}} = \inf_{\substack{P:\Q_p^n \to \Q_p^n \text{ rank $k$ projection}\\ W \subset \Q_p^m: \dim W = k}} \val_p(\det(P \diag(p^{\la_1},\ldots,p^{\la_n})|_W)).
\end{equation}
The infimum on the right hand side is clearly achieved by taking $W = \Span(\vec{e_{n-k+1}},\ldots,\vec{e_n})$ (where $\vec{e_i}$ are the standard basis vectors) and $P$ to be the projection onto $\Span(\vec{e_{n-k+1}},\ldots,\vec{e_n})$. This proves \eqref{eq:raleigh}.

\end{proof}

It turns out that one does not have to work with arbitrary projections and subspaces, but may instead consider only minors of the matrix $A$. Here by $k \times k$ minor, we mean any $k \times k$ matrix obtained by deleting rows and columns of the original matrix (and do not, as is also standard, mean the determinant of such a matrix). Though this version is slightly cumbersome to prove, it makes it much easier to relate matrix entries to singular numbers, and we expect it to be useful in future work as well as in this one.

\begin{prop}\label{thm:submatrices_suffice}
With the same setup as \Cref{thm:minmax},
\begin{equation}\label{eq:raleigh_submat}
\la_n+\ldots+\la_{n-k+1} = \inf_{A' \text{ $k \times k$ minor of $A$}} \val_p(\det(A')).
\end{equation}
\end{prop}
\begin{proof}
% Let $P,W$ be a projection and subspace realizing the infimum in \eqref{eq:raleigh}. Let $V_1,V_2 \in \GL_n(\Z_p)$ be such that $V_1PV_2 = \diag(1[k],0[n-k])$, and $U \in \GL_m(\Z_p)$ be such that $UW = \Span(\vec{e_1},\ldots,\vec{e_k})$. Then setting $A' = (a'_{i,j})_{\substack{1 \leq i \leq n \\ 1 \leq j \leq m}} := V_2^{-1}AU^{-1}$, we have
% \begin{align}
% \det(PA|_W) &= \det(V_1^{-1})\det(((V_1PV_2)(V_2^{-1}AU^{-1})U)_W ) \\ 
% &= \det((a'_{i,j})_{1 \leq i,j \leq k})
% \end{align}

% \fix{fill in, don't remember rn}
Clearly
\begin{equation}\label{eq:minor_diag}
 \la_n+\ldots+\la_{n-k+1} = \inf_{A' \text{ $k \times k$ minor of $\diag(p^{\la_1},\ldots,p^{\la_n})$}} \val_p(\det(A')).
\end{equation}
Since $U_1AU_2 = \diag(p^{\la_1},\ldots,p^{\la_n})$ for some $U_1 \in \GL_m(\Z_p),U_2 \in \GL_m(\Z_p)$, to show equality of the right hand side of \eqref{eq:minor_diag} and \eqref{eq:raleigh_submat} it therefore suffices to show that
\begin{equation}
\label{eq:minor_invariant}
\inf_{A' \text{ $k \times k$ minor of $B$}} \val_p(\det(A')) = \inf_{A' \text{ $k \times k$ minor of $UBV$}} \val_p(\det(A'))
\end{equation}
for any $B = (b_{i,j})_{\substack{1 \leq i \leq n \\ 1 \leq j \leq m}} \in \Mat_{n \times m}(\Q_p)$ and $U \in \GL_n(\Z_p),V \in \GL_m(\Z_p)$. 
% Slightly more simply, it suffices to show
% \begin{equation}
% \label{eq:ineq_minor_invariant}
% \inf_{A' \text{ $k \times k$ minor of $B$}} \val_p(\det(A')) \leq \inf_{A' \text{ $k \times k$ minor of $UBV$}} \val_p(\det(A')),
% \end{equation}
% since we may replace $B$ by $UBV$ and $U,V$ by $U^{-1},V^{-1}$ for the reverse inequality.

First note that since $\GL_n(\Z_p)$ is generated by the three elementary row operations
\begin{enumerate}
\item[(i)] elementary transposition matrices, 
\item[(ii)] unit multiple matrices $\diag(1[i-1],u,1[n-i])$ for $u \in \Z_p^\times, 1 \leq i \leq n$, and
\item[(iii)] matrices $(\bbone(i=j) + \bbone(i=x,j=y))_{1 \leq i,j \leq n}$ for some $x \neq y$,
\end{enumerate}
it suffices to prove \eqref{eq:minor_invariant} when $U$ and $V$ are each one of the above types. This is clear for types (i) and (ii). Suppose that one of $U,V$ is of type (iii), without loss of generality $U$ is of type (iii) and $V$ is the identity. Then 
\begin{equation}
UBV = (b_{i,j} + \bbone(i=x)b_{y,j})_{\substack{1 \leq i \leq n \\ 1 \leq j \leq m}}
\end{equation}
differs from $B$ only in the $x\tth$ row. For any two sets of indices $I_x=\{x,i_1,\ldots,i_{k-1}\}$ and $J=\{j_1,\ldots,j_k\}$ which include the row $x$, let $B_{I_x,J}$ be the corresponding minor. Then
\begin{equation}
\det (UB)_{I_x,J} = \det B_{I_x,J} + \det B_{I_y,J}
\end{equation}
so by the ultrametric inequality
\begin{equation}\label{eq:ultrametric}
\val_p(\det (UB)_{I_x,J} ) \geq \min(\val_p(\det B_{I_x,J}),\val_p(\det B_{I_y,J})).
\end{equation}
The $\leq$ direction of \eqref{eq:minor_invariant} follows immediately. For the $\geq$ direction, suppose that the strict equality case of \eqref{eq:ultrametric} holds. It is a standard fact about $\Q_p$ that if strict inequality in \eqref{eq:ultrametric} is achieved, then $\val_p(\det B_{I_x,J}) = \val_p(\det B_{I_y,J})$, so since $(UB)_{I_y,J} = B_{I_y,J}$ we have
\begin{equation}
\val_p(\det (UB)_{I_y,J}) = \val_p(\det B_{I_y,J}) = \val_p(\det B_{I_x,J}).
\end{equation}
It follows that any value achieved by the infimum on the left hand side of \eqref{eq:minor_invariant} must also be achieved by the one on the right hand side, and this proves \eqref{eq:minor_invariant}.
\end{proof}

\begin{rmk}
As mentioned, \Cref{thm:minmax} is a straightforward $p$-adic analogue of the corresponding statement \eqref{eq:sv_minmax} for complex matrices. However, the analogue of \Cref{thm:submatrices_suffice} for complex matrices, namely that products of singular values are related to an infimum over $k \times k$ minors, is manifestly false. As is apparent from the above proof, specific properties of the $p$-adic numbers such as ultrametricity are required to make the reduction from an infimum over all subspaces in \Cref{thm:minmax} to an infimum only over subspaces generated by subsets of the standard basis (and similarly for projections) in \Cref{thm:submatrices_suffice}.
\end{rmk}

We record a few other simple facts about singular numbers which will be useful.

\begin{prop}
\label{thm:multiply_smaller_dimension}
Let $n \leq m$, $A \in \Mat_{n \times m}(\Q_p)$, and $\kappa \in \Sig_m$. Then 
\[
|\SN(\diag(p^{\kappa_1},\ldots,p^{\kappa_n}) A)| = |\SN(A)| + |\kappa|.
\]
\end{prop}
\begin{proof}
%Follows immediately from \Cref{thm:minmax} with $k=n$ and multiplicativity of the determinant.
In the case $m=n$ this follows immediately since $\det(p^\kappa A) = \det(p^\kappa) \det(A)$. In general, $A$ is equivalent by column operations to a matrix $A'$ with nonzero entries only in the left $m \times m$ submatrix $A''$; clearly $\SN(A)=\SN(A'')$. Since column operations commute with left-multiplication by $p^\kappa$, we have $\SN(p^\kappa A) = \SN(p^\kappa A'')$, so we may appeal to the square case.
\end{proof}

\begin{prop}\label{thm:increase_decrease}
Let $n \leq m$, $A \in \Mat_{n \times m}(\Q_p)$, and suppose $B \in \Mat_m(\Q_p)$ has all singular numbers nonnegative (resp. nonpositive). Then $\SN(AB)_i \geq \SN(A)_i$ (resp. $\SN(AB)_i \leq \SN(A)_i$) for each $1 \leq i \leq n$. If $C \in \Mat_n(\Q_p)$ has nonnegative (resp. nonpositive) singular numbers, the same holds with $CA$ replacing $AB$.
\end{prop}
\begin{proof}
Write $B = UDV$ where $D = \diag(p^{\SN(B)})$. Then $\SN(A) = \SN(AU)$ and $\SN(AB) = \SN(AUD)$. Each minor determinant of $AUD$ is a nonnegative (resp. nonpositive) power of $p$ times the corresponding minor determinant of $AU$, so the desired inequality follows from \Cref{thm:submatrices_suffice}. The proof for $CA$ is the same.
\end{proof}

\section{Classifying isotropic processes}\label{sec:classify_processes}

In this section we prove \Cref{thm:classify_processes_intro}, by deducing it from a result (\Cref{thm:classify_processes_cts_G/K}) which translates the constraints of isotropy and stationarity of processes on $\GL_N(\Q_p)/\GL_N(\Z_p)$ into a usable form. For expository purposes, we first prove a version in discrete time (\Cref{thm:classify_discrete_processes}) which makes the basic ideas of \Cref{thm:classify_processes_cts_G/K} slightly more apparent.

\begin{defi}
A stochastic process $X(\tau), \tau \in \Z_{\geq 0}$ on a group $G$ has \emph{independent increments} if for any $s,\tau \in \Z_{\geq 0}$, $X(\tau+s)X(\tau)^{-1}$ is independent of the trajectory of $X(y)$ up to time $\tau$. It has \emph{stationary increments} if 
\begin{equation}
\Law(X(\tau+s)X(\tau)^{-1}) = \Law(X(s)X(0)^{-1})
\end{equation}
for all such $s,\tau$. For a subgroup $K \leq G$, we further say $X(\tau)$ has \emph{$K$-isotropic increments} if 
\begin{equation}
\Law(X(\tau+s)X(\tau)^{-1}) = \Law(kX(\tau+s)X(\tau)^{-1}k^{-1})
\end{equation}
for any $k \in K, s, \tau \geq 0$. We use the same terminology for continuous-time processes with $\Z_{\geq 0}$ replaced everywhere by $\R_{\geq 0}$.
\end{defi}

% \fix{remove def if not needed}
% \begin{defi}
% A stochastic process $X(\tau)$ on $\GL_N(\Q_p)$ is \emph{(left-) $\GL_N(\Z_p)$-invariant} if for any $U \in \GL_N(\Z_p)$ and $\tau \geq 0$, $UX(\tau) = X(\tau)$ in distribution. Right-invariance is defined as expected.
% \end{defi}

% \begin{defi}\label{def:disc_processes}
% We denote by $\MKD(N)$ the set of $\GL_N(\Z_p)$-invariant discrete-time Markov processes $X(\tau), \tau \in \Z_{\geq 0}$ on $\GL_N(\Q_p)/\GL_N(\Z_p)$ which have stationary independent increments.
% \end{defi}

% \begin{defi}\label{def:cts_processes}
% We denote by $\MKC(N)$ the set of continuous-time $\GL_N(\Z_p)$-invariant Markov processes $X(\tau), \tau \in \R_{\geq 0}$ on $\GL_N(\Q_p)/\GL_N(\Z_p)$ which have stationary independent incrememts.
% % \begin{equation}\label{eq:no_blowup}
% % 0 <  \E[\inf\{s > 0: X(s) \neq X(0)\}] < \infty.
% % \end{equation}
% \end{defi}

We now state and prove the discrete-time result, which follows directly from the definitions. %This is certainly known to experts, see e.g. \cite[(0.1)]{parkinson2007isotropic} for a very similar general result, though we are not aware of a reference for precisely this version.

\begin{prop}\label{thm:classify_discrete_processes}
Let $X(\tau), \tau \in \Z_{\geq 0}$ be a discrete-time stochastic process on $\GL_N(\Q_p)$ started at the identity, with stationary, independent, $\GL_N(\Z_p)$-isotropic increments, and set $M^{(d)}_X := \Law(\SN(X(1)X(0)^{-1}))$. Then there exists a distribution on triples $(U,V,\nu)$, such that the marginal distribution of each pair $(U,\nu)$ and $(V,\nu)$ is $M_{Haar}(\GL_N(\Z_p)) \times M^{(d)}_X$, for which
\begin{equation}\label{eq:disc_wts}
\Law(X(\tau), \tau \in \Z_{\geq 0}) = \Law( U_{\tau} \diag(p^{\nu^{(\tau)}_1},\ldots,p^{\nu^{(\tau)}_N}) V_{\tau}  \cdots U_1 \diag(p^{\nu^{(1)}_1},\ldots,p^{\nu^{(1)}_N}) V_1 , \tau \in \Z_{\geq 0})
\end{equation}
where $(U_i,V_i,\nu^{(i)})$ are iid copies of $(U,V,\nu)$. 
%where $X^{(i)}$ are iid copies of $X(1)X(0)^{-1}$.
% \begin{equation}
% \hX(\tau) := [U_\tau \diag(p^{\nu(\tau)}) U_{\tau-1} \diag(p^{\nu(\tau-1)}) \cdots U_1 \diag(p^{\nu^{(1)}})]
% \end{equation}
% for iid signatures $\nu^{(i)} \sim M^{(d)}_X, i \in \Z_{\geq 0}$ and iid $\U(i) \sim M_{Haar}(\GL_N(\Z_p))$. Then $\hX(\tau) \in \MKD(N)$ and
% \begin{equation}\label{eq:hX_disc_works}
% \Law(\SN(X(\tau)),\tau \in \Z_{\geq 0}) = \Law(\SN(\hX(\tau)),\tau \in \Z_{\geq 0}).
% \end{equation}
\end{prop}

\begin{rmk}\label{rmk:not_indep}
We note that while the pairs $(U,\nu)$ and $(V,\nu)$ are each distributed by product measures, the pair $(U,V)$ need not be. For example, one may have $U=V^{-1} \sim M_{Haar}(\GL_N(\Z_p))$. Of course, $U$ and $V$ can also be independent Haar matrices.
\end{rmk}

\begin{proof}[Proof of \Cref{thm:classify_discrete_processes}]

% \begin{proof}
% By Kolmogorov's extension theorem it suffices to show for any fixed $\tau \in \Z_{\geq 0}$ that 
% \begin{equation}\label{eq:finite_disc_law_match}
% \Law(X(s), 0 \leq s \leq \tau) = \Law( U_s \diag(p^{\nu(s)}) U_{s-1} \cdots U_1 \diag(p^{\nu^{(1)}}) , 0 \leq s \leq \tau).
% \end{equation}
% We induct on $\tau$, the base case $\tau=0$ being trivial, so suppose the claim holds for some $\tau$. Then \fix{we should go define the M sub Haar notation early, figure out where}
Consider the increment
\begin{equation}
X(\tau+1) = (X(\tau+1)X(\tau)^{-1})X(\tau)
\end{equation}
corresponding to the time step $\tau \to \tau+1$. By the independent increments property $X(\tau+1)X(\tau)^{-1}$ is independent of $X(0),\ldots,X(\tau)$ and distributed as $X(1)X(0)^{-1}$. Hence $X(\tau+1)X(\tau)^{-1} = WDV$ for $W,V \in \GL_N(\Z_p)$ and $D = \diag(p^{\nu(\tau+1)})$ with $\nu(\tau+1) \sim M_X^{(d)}$ by definition of $M_X^{(d)}$ and stationary increments, and all of these are independent of $X(0),\ldots,X(\tau)$. By isotropy, 
\begin{equation} \label{eq:use_isotropy}
W D V = \tU W D V \tU^{-1}
\end{equation}
in distribution for any fixed $\tU \in \GL_N(\Z_p)$. Hence by averaging, \eqref{eq:use_isotropy} also holds when $\tU$ is random with Haar distribution independent of $W,D,V$ and $X(0),\ldots,X(\tau)$. Because $\tU$ is Haar-distributed independent of $D$ and of $W$ and $V$, $\tU W$ and $V\tU^{-1}$ are each Haar-distributed independent of $D$. Defining $U_{\tau+1} = \tU W$ and $V_{\tau+1} = V \tU^{-1}$, we thus have that the increments are of the form in the right hand side of \eqref{eq:disc_wts}, which completes the proof.

%Defining $U_{\tau+1} = \tU W$ and $V_{\tau+1} = V \tU^{-1}$, we thus have that $U_{\tau+1}, V_{\tau+1} \sim M_{Haar}(\GL_N(\Z_p))$ independent $X(0),\ldots,X(\tau)$, and $V_{\tau+1}$ is independent of $\nu^{(\tau+1)}$. Hence the increments are of the form in the right hand side of \eqref{eq:disc_wts}, which completes the proof.

% Thus inductive hypothesis,
% \begin{align}\label{eq:finish_disc_case}
% \begin{split}
% &\Law(X(s), 0 \leq s \leq \tau+1) \\
% &= \Law(X(0), \ldots, X(\tau),  U_{\tau+1} \diag(p^{\nu^{(\tau+1)}}) V_{\tau+1} X(\tau)) \\ 
% &= \Law(\Id,U_1 \diag(p^{\nu^{(1)}}) V_1, \ldots, \SN(\tU WDV \tU^{-1} U_\tau \diag(p^{\nu(\tau)}) U_{\tau-1} \cdots \diag(p^{\nu^{(1)}}) U_1^{-1} \cdots U_\tau^{-1} ). %\\ 
% %&= \Law(\SN( \diag(p^{\nu(s)}) U_s \cdots \diag(p^{\nu^{(1)}}) U_1 X(0)), 0 \leq s \leq \tau+1).
% \end{split}
% \end{align}
% By conjugating each increment 
% \fix{incomplete}

% Here in the last line we set $U_{\tau+1} := V\tU^{-1}$ and note that it is independent of $D$ and of $\SN(\diag(p^{\nu(s)}) U_s \cdots \diag(p^{\nu^{(1)}}) U_1 X(0)), 0 \leq s \leq \tau$, and note also that since $\tU W$ lies in $\GL_N(\Z_p)$, its distribution (whatever it is) does not affect the singular numbers. This completes the induction.
\end{proof}

% The main difference in the continuous-time version is that we pass to $\GL_N(\Q_p)/\GL_N(\Z_p)$ to keep working on a discrete space.

We now explicitly define the measure and Poisson rate constant claimed to exist in \Cref{thm:classify_processes_intro}. We will work on the homogeneous space $\GL_N(\Q_p)/\GL_N(\Z_p)$, since it is discrete and so all processes on it are Poisson jump processes.

\begin{defi}\label{def:cosets}
For any $X \in \GL_N(\Q_p)$, we denote by $[X]$ the corresponding coset in $\GL_N(\Q_p)/\GL_N(\Z_p)$.
\end{defi}

\begin{defi}\label{def:mx_and_c}
Given a stochastic process $X(\tau), \tau \in \R_{\geq 0}$ on $\GL_N(\Q_p)$ satisfying the conditions of \Cref{thm:classify_processes_intro}, we define
\begin{align}
\begin{split}
\tau' &= \inf\{\tau > 0: [X(\tau)] \neq [X(0)]\} \\ 
M_X &= \Law(\SN(X(\tau'))) \\ 
c &= \E[\tau'].
\end{split}
\end{align}
\end{defi}

\begin{prop}\label{thm:classify_processes_cts_G/K}
 Let $N \in \mathbb{Z}_{\geq 1}$ and let $X(\tau), \tau \in \mathbb{R}_{\geq 0}$ be a Markov process on $\mathrm{GL}_{N}(\mathbb{Q}_{p})$ started at the identity with stationary, independent, $\GL_N(\Z_p)$-isotropic increments. Then there exists a distribution on triples $(U,V,\nu)$, such that the marginal distribution of each pair $(U,\nu)$ and $(V,\nu)$ is $M_{Haar}(\GL_N(\Z_p)) \times M_X$, for which
\begin{equation}\label{eq:cts_wts}
\Law([X(\tau)], \tau \in \R_{\geq 0}) = \Law([U_{P(\tau)} \diag(p^{\nu^{(P(\tau))}_1},\ldots,p^{\nu^{(P(\tau))}_N}) V_{P(\tau)}  \cdots U_1 \diag(p^{\nu^{(1)}_1},\ldots,p^{\nu^{(1)}_N}) V_1] , \tau \in \R_{\geq 0}),
\end{equation}
where $(U_i,V_i,\nu^{(i)})$ are iid copies of $(U,V,\nu)$ and $P(\tau)$ is a Poisson process whose rate is the constant $c$ in \Cref{def:mx_and_c}.
% \begin{equation}
% [X(\tau)] = [U_{P(\tau)} D_{P(\tau)} V_{P(\tau)} U_{P(\tau)-1} D_{P(\tau)-1} V_{P(\tau)-1} \cdots U_1 D_1 V_1 ] \quad \quad \quad \quad \text{ in multi-time distribution,}
% \end{equation}
% where $U_i, i \geq 0$ are drawn independently from the Haar measure on $\GL_N(\Z_p)$, $D_i = \diag(p^{\nu^{(i)}_1},\ldots,p^{\nu^{(i)}_N})$ with $\nu^{(i)} \sim M_X$ iid, and $V_i \in \GL_N(\Z_p)$ is independent of all matrices besides $D_i$ and $U_i$. This directly implies that
\end{prop}
\begin{proof}
First note that the dynamics of $X(\tau)$ commutes with right-multiplication by $\GL_N(\Z_p)$ (in fact, by $\GL_N(\Q_p)$), so $[X(\tau)]$ is Markov. Define the $\Z_{\geq 0}$-valued process
\begin{equation}
N_X(\tau) = |\{0 < s \leq \tau: [X(s)] \neq \lim_{\eps \to 0^+} [X(s-\eps)] \}|,
\end{equation}
i.e. the number of times $[X(\tau)]$ has changed value up to time $\tau$. By the Markov property and stationary increments, $N_X(\tau)$ is a Poisson process $P(\tau)$ with rate $c$ as defined in \Cref{def:mx_and_c}. Let $t_1 < t_2 < \ldots$ be the (random) times realizing $\SN(X(\tau)) \neq \lim_{\eps \to 0^+} \SN(X(\tau-\eps))$ and $t_0 = 0$, so $X(\tau) = X(t_{N_X(\tau)})$ and this is equal in distribution to $X(t_{P(\tau)})$. By Kolmogorov's extension theorem it suffices to show 
\begin{equation}\label{eq:law_eq_with_P}
\Law([X(t_{P(s)})], 0 \leq s \leq \tau) = \Law([U_{P(s)}\diag(p^{\nu(P(s))}) V_{P(s)}  \cdots U_1 \diag(p^{\nu^{(1)}})V_1, 0 \leq s \leq \tau)
\end{equation}
for any fixed $\tau \in \R_{\geq 0}$. It further suffices to show the equality of conditional laws
\begin{multline}\label{eq:law_eq_with_P_cond}
\Law([X(t_{P(s)})], 0 \leq s \leq \tau| P(s),0 \leq s \leq \tau ) \\ 
= \Law([U_{P(s)}\diag(p^{\nu(P(s))}) V_{P(s)}  \cdots U_1 \diag(p^{\nu^{(1)}})V_1], 0 \leq s \leq \tau | P(s),0 \leq s \leq \tau ),
\end{multline}
as then one may average over the distribution of $P(s), 0 \leq s \leq \tau$ to obtain \eqref{eq:law_eq_with_P}. By the strong Markov property, $X(t_{i+1})X(t_i)^{-1}$ is independent of $X(\tau), 0 \leq \tau \leq t_i$ and independent of $t_0,\ldots,t_i$, hence it suffices to show 
\begin{equation}\label{eq:law_eq_with_P_cond_disc}
\Law([X(t_{i})], 0 \leq i \leq n ) 
= \Law([U_{i}\diag(p^{\nu(i)}) V_{i}  \cdots U_1 \diag(p^{\nu^{(1)}})V_1], 0 \leq i \leq n )
\end{equation}
for all $n \in \Z_{\geq 1}$. Since each increment $X(t_{i+1})X(t_i)^{-1}$ is distributed as $X(t_1)X(0)^{-1}$ by independent increments, \eqref{eq:law_eq_with_P_cond_disc} is exactly the discrete case \Cref{thm:classify_discrete_processes} and we are done.
\end{proof}

\begin{proof}[Proof of \Cref{thm:classify_processes_intro}]
Since Smith normal form is independent of right-multiplication by $\GL_N(\Z_p)$, we may write $\SN([X])$ for $[X] \in \GL_N(\Q_p)/\GL_N(\Z_p)$ with no ambiguity. By \Cref{thm:classify_processes_cts_G/K},
\begin{equation}
\SN(X(\tau)) = \SN([X(\tau)]) = \SN([\tU_{P(\tau)} \diag(p^{\tnu^{(P(\tau))}_1},\ldots,p^{\tnu^{(P(\tau))}_N}) \tV_{P(\tau)}  \cdots \tU_1 \diag(p^{\tnu^{(1)}_1},\ldots,p^{\tnu^{(1)}_N}) \tV_1])
\end{equation}
in multi-time distribution, where $\tU_i,\tV_i, \tnu^{(i)}$ correspond to the $U_i,V_i,\nu^{(i)}$ in \Cref{thm:classify_processes_cts_G/K}. We write them with the tildes to distinguish them from 
\begin{equation}
Y^{(N,M_X,c)}(\tau) = U_{P(\tau)} \diag(p^{\nu^{(P(\tau))}_1},\ldots,p^{\nu^{(P(\tau))}_N}) V_{P(\tau)}  \cdots U_1 \diag(p^{\nu^{(1)}_1},\ldots,p^{\nu^{(1)}_N}) V_1 ,
\end{equation}
for which $U_i$ and $V_i$ are independent as we recall from \Cref{def:the_processes}. As in the proof of \Cref{thm:classify_processes_cts_G/K} we are reduced to showing that 
\begin{multline}\label{eq:law_eq_sn}
\Law(\SN(U_{i}\diag(p^{\nu(i)}) V_{i}  \cdots U_1 \diag(p^{\nu^{(1)}})V_1), 0 \leq i \leq n )  \\
=\Law(\SN(\tU_{i}\diag(p^{\tnu(i)}) \tV_{i}  \cdots \tU_1 \diag(p^{\tnu^{(1)}})\tV_1), 0 \leq i \leq n ),
\end{multline}
which we do by induction. The base case is trivial, so assume it holds for some $n$. Then by inductive hypothesis,
\begin{align}\label{eq:ind_tilde_joint_law}
\begin{split}
&\Law(\SN(\tU_{i}\diag(p^{\tnu(i)}) \tV_{i}  \cdots \tU_1 \diag(p^{\tnu^{(1)}})\tV_1), 0 \leq i \leq n +1) \\ 
%&= \Law(\SN(\tU_{i}\diag(p^{\tnu(i)}) \tV_{i}  \cdots \tU_1 \diag(p^{\tnu^{(1)}})\tV_1), 0 \leq i \leq n; \SN(\tU_{n+1}\diag(p^{\tnu(n+1)}) \tV_{n+1} U_n \diag(p^{\nu^{(n)}}) V_n \cdots U_1\diag(p^{\nu^{(1)}})V_1))
&=\Law(\SN(U_1\diag(p^{\nu^{(1)}})V_1),\ldots,\SN(U_n \diag(p^{\nu^{(n)}}) V_n \cdots U_1\diag(p^{\nu^{(1)}})V_1), \\ 
&\SN(\diag(p^{\tnu^{(n+1)}}) \tV_{n+1} U_n \diag(p^{\nu^{(n)}}) V_n \cdots U_1\diag(p^{\nu^{(1)}})V_1)),
\end{split}
\end{align}
where we have removed the $\tU_{n+1}$ on the left since it does not affect the singular numbers. Because $(\tnu^{(n+1)},\tV_{n+1}) \sim M_X \times M_{Haar}(\GL_N(\Z_p))$ and $(\nu^{(n+1)},V_{n+1}) \sim M_X \times M_{Haar}(\GL_N(\Z_p))$ as well,
\begin{equation}
\text{RHS\eqref{eq:ind_tilde_joint_law}} = \Law(\SN(U_{i}\diag(p^{\nu(i)}) V_{i}  \cdots U_1 \diag(p^{\nu^{(1)}})V_1), 0 \leq i \leq n +1)
\end{equation}
(adding back in the $U_{n+1}$ which does not affect the singular numbers). This completes the induction to show \eqref{eq:law_eq_sn} and hence the proof.

\end{proof}

\begin{rmk}\label{rmk:any_group}
The above results and proofs apply mutatis mutandis with the groups $\GL_N(\Z_p) \leq \GL_N(\Q_p)$ replaced by any groups $K \leq G$ with $K$ compact and $G/K$ discrete, and $\Sig_N$ replaced by $K\backslash G /K$.
\end{rmk}

We now turn attention to those processes with $M_X = \delta_{(1,0[N-1])}$. \Cref{thm:classify_processes_intro} implies that the evolution of singular numbers of such a process are determined by a rate parameter which may be absorbed by time change. We refer to the discussion directly after \Cref{thm:classify_processes_intro} for why these are natural analogues of multiplicative Brownian motion.

\begin{defi}\label{def:our_process}
For any $N \in \Z_{\geq 1}$, we define a continuous-time stochastic process $X^{(N)}(\tau)$ on $\GL_N(\Q_p)$ by
\begin{equation}
X^{(N)}(\tau) = Y^{(N,\delta_{(1,0[N-1])},1)}(\tau) = U_{P(\tau)} \diag(p,1[N-1])V_{P(\tau)} \cdots U_1 \diag(p,1[N-1])V_1,
\end{equation}
where $P(\tau)$ is a rate $1$ Poisson process and $U_i,V_i, i \in \Z_{\geq 1}$ are independent matrices distributed by the Haar measure on $\GL_N(\Z_p)$.
\end{defi}

\section{$p$-adic Dyson Brownian motion and reflected Poisson walks}\label{sec:poisson_walks}

%plan: first show matrix thing relates to HL structure coefs, then show this leads to HL process, then show this leads to poisson walks

In this section we prove \Cref{thm:dbm_poisson_intro}. Much of this section consists of computing and comparing Markov generators. We wish to go from equalities of generators to equalities of stochastic processes, for which we use the following standard result of Feller \cite{feller2015integro}, see also Borodin-Olshanski \cite[Section 4.1]{borodin2012markov} which in addition give stronger results than we need:

\begin{prop}\label{thm:cite_feller}
Let $Y_\tau, \tau \in \R_{\geq 0}$ be a Markov process on a countable state space $\mathcal{X}$ with well-defined generator $Q$ satisfying
\begin{equation}
\sup_{a \in \mathcal{X}} |Q(a,a)| < \infty.
\end{equation}
Then the law of $Y_\tau, \tau \in \R_{\geq 0}$ is uniquely determined by $Q$. 
\end{prop}

We now give the generator for $\Pois$, which in light of \Cref{thm:cite_feller} serves as an alternative and more formal definition. 

\begin{defi}
    \label{def:poisson_walks_generator}
    Let $n \in \N$, $\mu \in \Sig_{n}$ and $t \in (0,1)$. We define the stochastic process $\Pois^{\mu,n}(\tau) = (\Pois^{\mu,n}_1(\tau),\ldots,\Pois^{\mu,n}_n(\tau))$ on $\Sig_n$ as the Markov process with initial condition $\mu$ and generator given by\footnote{We suppress $n$ and $t$ in the notation for the generator, but of course it depends on both.} 
    \begin{equation}\label{eq:poisson_gen}
    B_{\Pois}(\kappa,\nu) := \begin{cases} 
-t\frac{1-t^n}{1-t} & \kappa = \nu \\ 
t^{\ell} + \ldots + t^{\ell+m_{\kappa_\ell}(\kappa)-1} = t^\ell \frac{1-t^{m_{\kappa_\ell}(\kappa)}}{1-t} & \kappa \prec \nu \text{ and } (\nu_i)_{1 \leq i \leq n} = (\kappa_i + \bbone(i=\ell))_{1 \leq i \leq n} \\
0 & \text{otherwise}
\end{cases}.
    \end{equation}
    We further allow $n=\infty$, in which case we take $t^n=0$ in the above formulas, and may have $m_{\kappa_\ell}(\kappa) = \infty$.
\end{defi}

This is equivalent to the description in \Cref{def:poisson_walks} since the off-diagonal entries $B_\Pois(\kappa,\nu)$ are simply the sums of the jump rates of all indices $\Pois_i^{\mu,n}$ such that the ringing of the clock associated to that index will move $\Pois^{\mu,n}$ from $\kappa$ to $\nu$ by the rules of \Cref{def:poisson_walks}, as the reader may easily convince themselves. The following result is the main computation for \Cref{thm:dbm_poisson_intro}.

\begin{prop}\label{thm:compute_single_prod}
Let $N \in \N$, $\kappa \in \Sig_N$, and let $1 \leq \ell \leq N$ be such that $\nu := \kappa + \vec{e}_\ell$ lies in $\Sig_N$. Then for a Haar-distributed element $U$ of $\GL_N(\Z_p)$,
\begin{equation}\label{eq:1box_prod_prob}
\Pr(\SN(\diag(p,1[N-1])U\diag(p^\kappa)) = \nu) = \frac{1-t}{1-t^N} t^{\ell-1} (1-t^{m_{\kappa_\ell}(\kappa)}),
\end{equation}
where as usual $t=1/p$.
\end{prop}

We require a key linear-algebraic lemma.

\begin{lemma}\label{thm:which_is_unit}
Let $A = (a_{i,j})_{1 \leq i,j \leq N} \in \GL_N(\Z_p)$ and $\kappa \in \Sig_N$. Let $\tell = \sup \{1 \leq i \leq N: a_{i,1} \in \Z_p^\times\}$ (note that the set is nonempty since $A \in \GL_N(\Z_p)$), and $\ell = \sup \{i: \kappa_i = \kappa_\tell\}$. Then 
\begin{equation}\label{eq:1jump}
\SN(\diag({p^{\kappa}})A\diag(p^{-1},1,\ldots,1)) = \kappa - \vec{e}_\ell.
\end{equation}
\end{lemma}
\begin{proof}
Let $\nu$ be the signature on the left hand side of \eqref{eq:1jump}. We first show that if $\ell < N$ then
\begin{equation}\label{eq:latter_sns_equal}
(\nu_{\ell+1},\ldots,\nu_N) = (\kappa_{\ell+1},\ldots,\kappa_N).
\end{equation}
For $\ell+1 \leq i \leq N$, all entries in the $i\tth$ row of 
\begin{equation}
\diag({p^{\kappa}})A\diag(p^{-1},1,\ldots,1) = \begin{pmatrix}
p^{\kappa_1-1}a_{1,1} & p^{\kappa_1}a_{1,2} & \cdots & p^{\kappa_1}a_{1,N} \\ 
p^{\kappa_2-1}a_{2,1} & p^{\kappa_2}a_{2,2} & \cdots & p^{\kappa_2}a_{2,N} \\ 
\vdots & \vdots & \ddots & \vdots \\ 
p^{\kappa_N-1}a_{N,1} & p^{\kappa_N}a_{N,2} & \cdots & p^{\kappa_N}a_{N,N} 
\end{pmatrix}
\end{equation}
lie in $p^{\kappa_i}\Z_p$: this is manifestly true for all entries but the first, and for the first entry $p^{\kappa_i-1}a_{i,1}$ it follows because $a_{i,1} \in p\Z_p$ since $i > \ell$. Furthermore, all entries in the top $\ell$ rows lie in $p^{\kappa_\ell-1}\Z_p \subset p^{\kappa_{\ell+1}}\Z_p$, using that $\kappa_\ell > \kappa_{\ell+1}$. Hence for any $(N-\ell) \times (N-\ell)$ minor, all entries in its bottom row lie in $p^{\kappa_N}\Z_p$, all in its second from the bottom row lie in $p^{\kappa_{N-1}}\Z_p$, etc., regardless of which rows are chosen. Thus its determinant will be a sum of products of $N-\ell$ terms lying in $p^{\kappa_N}\Z_p, p^{\kappa_{N-1}}\Z_p,\ldots,p^{\kappa_{\ell+1}}\Z_p$, and hence will lie in $p^{\kappa_{\ell+1}+\ldots+\kappa_N}\Z_p$. Since this is true for every minor, \Cref{thm:submatrices_suffice} implies that $\nu_{\ell+1}+\ldots+\nu_N \geq \kappa_{\ell+1}+\ldots+\kappa_N$. Since $\nu_i \leq \kappa_i$ for all $i$ by \Cref{thm:increase_decrease}, it follows that $\nu_i = \kappa_i$ for $\ell+1 \leq i \leq N$. Running the same argument with $\ell+1$ replaced by any index $j$, so that the first entry of the $j\tth$ row no longer necessarily lies in $p^{\kappa_i}\Z_p$, we still obtain the bound
\begin{equation}\label{eq:nu_lb}
\sum_{i=j}^N \nu_i \geq \sum_{i=j}^N \kappa_i - 1,
\end{equation}
which will be useful later.

Because $A \in \GL_N(\Z_p)$, rows $\tell,\ell+1,\ell+2,\ldots,N$ are linearly independent modulo $p$. The first entry $a_{\tell,1}$ of row $\tell$ lies in $\Z_p^\times$, and the first entries of rows $\ell+1,\ldots,N$ lie in $p\Z_p$, by the definition of $\tell$. Hence the matrix $(a_{i,j})_{\substack{\ell+1 \leq i \leq N \\ 2 \leq j \leq N}}$ is full-rank modulo $p$, so there exist $j_1,\ldots,j_{N-\ell}$ such that the minor of this matrix determined by columns $j_1,\ldots,j_{N-\ell}$ has determinant in $\Z_p^\times$. Since $a_{\tell,1} \in \Z_p^\times$ and $a_{\ell+1,1},\ldots,a_{N,1} \in p\Z_p$, this implies that
\begin{equation}
A' :=
\begin{pmatrix}
a_{\tell,1} & a_{\tell,j_{1}} & \cdots & a_{\tell,j_{N-\ell}} \\ 
a_{\ell+1,1} & a_{\ell+1,j_{1}} & \cdots & a_{\ell+1,j_{N-\ell}} \\ 
\vdots & \vdots & \ddots & \vdots \\ 
a_{N,1} & a_{N,j_{1}} & \cdots & a_{N,j_{N-\ell}} \\ 
\end{pmatrix}
\end{equation}
is an $(N-\ell+1)\times (N-\ell+1)$ minor of $A$ with determinant lying in $\Z_p^\times$. By multiplicaivity of the determinant,
\begin{multline}\label{eq:det_submatrix}
\det(\diag(p^{\kappa_\tell},p^{\kappa_{\ell+1}},\ldots,p^{\kappa_N})A'\diag(p^{-1},1[N-\ell]))  \\ 
= \det 
\begin{pmatrix}
p^{\kappa_\tell - 1} a_{\tell,1} & p^{\kappa_\tell} a_{\tell,j_{1}} &\cdots & p^{\kappa_\tell} a_{\tell,j_{N-\ell}} \\ 
p^{\kappa_{\ell+1} - 1} a_{\ell+1,1} & p^{\kappa_{\ell+1}} a_{{\ell+1},j_{1}} &\cdots & p^{\kappa_{\ell+1}} a_{{\ell+1},j_{N-\ell}} \\ 
\vdots & \vdots & \ddots & \vdots \\ 
p^{\kappa_N - 1} a_{N,1} & p^{\kappa_N} a_{N,j_{1}} &\cdots & p^{\kappa_N} a_{N,j_{N-\ell}} \\ 
\end{pmatrix}  \in p^{\kappa_\tell - 1 + (\kappa_{\ell+1} + \ldots + \kappa_N)}\Z_p^\times.
\end{multline}
Because the matrix in \eqref{eq:det_submatrix} is a submatrix of $\diag({p^{\kappa}})A\diag(p^{-1},1,\ldots,1)$, it follows by \Cref{thm:submatrices_suffice} that 
\begin{equation}\label{eq:sum_differ_1}
\nu_\ell + \ldots + \nu_N \leq \kappa_\ell + \ldots + \kappa_N - 1.
\end{equation}
Combining \eqref{eq:latter_sns_equal}, \eqref{eq:sum_differ_1}, and \eqref{eq:nu_lb} (with $j=\ell$) yields that $\nu_\ell = \kappa_\ell - 1$. Because $\sum_i \nu_i = \sum_i \kappa_i - 1$ (which again follows by multiplicativity of the determinant) and $\nu_i \leq \kappa_i$ for all $i$ by \Cref{thm:increase_decrease}, \eqref{eq:1jump} follows.
\end{proof}

\begin{proof}[Proof of \Cref{thm:compute_single_prod}]
We invert the matrix in \eqref{eq:1box_prod_prob} to rewrite the probability in that equation as 
\begin{equation}\label{eq:inverted_SNs}
\Pr(\SN(\diag({p^{-\kappa_1},\ldots,p^{-\kappa_N}})V'\diag(p^{-1},1,\ldots,1)) = (-\nu_N,\ldots,-\nu_1)),
\end{equation}
where $V'=U^{-1}$ is also Haar distributed. Since the permutation matrix $\sigma = (\bbone(i+j=N+1))_{1 \leq i,j \leq N}$ lies in $\GL_N(\Z_p)$, we have
\begin{align}
\begin{split}
&\SN(\diag({p^{-\kappa_1},\ldots,p^{-\kappa_N}})V'\diag(p^{-1},1,\ldots,1)) \\ 
&= \SN(\sigma \diag({p^{-\kappa_1},\ldots,p^{-\kappa_N}}) \sigma (\sigma V')\diag(p^{-1},1,\ldots,1)) \\ 
&= \SN(\diag(p^{-\kappa_N},\ldots,p^{-\kappa_1})V\diag(p^{-1},1[N-1])) \\ 
&= \SN\left(\begin{pmatrix}
p^{-\kappa_N-1}v_{1,1} & p^{-\kappa_N}v_{1,2} & \cdots & p^{-\kappa_N}v_{1,N} \\ 
p^{-\kappa_{N-1}-1}v_{2,1} & p^{-\kappa_{N-1}}v_{2,2} & \cdots & p^{-\kappa_{N-1}}v_{2,N} \\ 
\vdots & \vdots & \ddots & \vdots \\ 
p^{-\kappa_1-1}v_{N,1} & p^{-\kappa_1}v_{N,2} & \cdots & p^{-\kappa_1}v_{N,N} 
\end{pmatrix}\right)
\end{split}
\end{align}
where $V=\sigma V'$ is also Haar-distributed. Define the (random) integers $\tell(V) = \inf \{1 \leq i \leq N: v_{i,1} \in \Z_p^\times\}$ and $\ell(V) = \sup \{i: -\kappa_{N-i+1} = -\kappa_{N-\tell+1}\}$ as in \Cref{thm:which_is_unit} (with $(-\kappa_N,\ldots,-\kappa_1)$ replacing the $\kappa$ in \Cref{thm:which_is_unit}). Then by \Cref{thm:which_is_unit} with $\kappa,\nu$ replaced by $(-\kappa_N,\ldots,-\kappa_1)$ and $(-\nu_N,\ldots,-\nu_1)$ respectively, we have 
\begin{equation}\label{eq:reduce_to_l(V)}
\Pr(\SN(\diag({p^{-\kappa_N},\ldots,p^{-\kappa_1}})V\diag(p^{-1},1,\ldots,1)) = (-\nu_N,\ldots,-\nu_1)) = \Pr(\ell(V) = N-\ell+1).
\end{equation}
Now we find the distribution of $\ell(V)$, which depends only on the first column $(v_{1,1},v_{2,1},\ldots,v_{N,1})^T$ modulo $p$. By \Cref{thm:haar_sampling}, $(w_1,\ldots,w_N) := (v_{1,1} \pmod{p},\ldots,v_{N,1} \pmod{p})$ is distributed as a uniformly random element of $\F_p^N \setminus 0$. In other words, $\bbone(w_i \neq 0), 1 \leq i \leq N$ is a collection of independent Bernoulli random variables with parameter $t$, conditioned on not all of them being $0$. An elementary computation yields the right hand side of \eqref{eq:1box_prod_prob} and completes the proof. 
\end{proof}

\begin{proof}[Proof of \Cref{thm:dbm_poisson_intro}]
Recall from \Cref{def:our_process} that 
\begin{equation}
X^{(N)}(\tau) := U_{P(\tau)} \diag(p,1[N-1]) V_{P(\tau)} \cdots U_1 \diag(p,1[N-1])V_1 ,
\end{equation}
where $P(\tau)$ is a rate-$1$ Poisson process and $U_1,\ldots, V_1,V_2,\ldots$ are independent Haar elements of $\GL_N(\Z_p)$. Hence the matrices $V_iU_{i-1}, i \geq 1$ are independent Haar matrices, so the process $\SN(X^{(N)})$ on state space $\Sig_N$ has generator
\begin{equation}
A(\kappa,\nu) = \begin{cases}
-1 & \nu = \kappa \\ 
\Pr_{U \sim M_{Haar}(\GL_N(\Z_p))}(\SN(\diag(p^\kappa)U\diag(p,1[N-1])) = \nu) & \text{else}
\end{cases}.
\end{equation}
Applying \Cref{thm:compute_single_prod} yields
\begin{equation}
A(\kappa,\nu) = \begin{cases}
-1 & \nu = \kappa \\ 
\frac{1}{1-t^N}t^{\ell-1} (1-t^{m_{\kappa_\ell}(\kappa)}) & \nu = \kappa + \vec{e}_\ell \\ 
0 & \text{else}
\end{cases}.
\end{equation}
Hence 
\begin{equation}\label{eq:equal_gens}
A(\kappa,\nu) = \frac{1}{t} \frac{1-t}{1-t^N} B_\Pois(\kappa,\nu)
\end{equation}
for all $\kappa,\nu \in \Sig_N$. \Cref{thm:dbm_poisson_intro} now follows from \eqref{eq:equal_gens} and \Cref{thm:cite_feller}.
\end{proof}

\end{document}